\documentclass[a4paper]{amsart}

\usepackage{amsthm}
\usepackage{amsmath}
\usepackage{amssymb}
\usepackage{enumerate}

\theoremstyle{plain} \newtheorem{theorem}{Theorem}[section]
\theoremstyle{plain} 
\theoremstyle{plain} \newtheorem{lemma}[theorem]{Lemma}
\theoremstyle{plain} \newtheorem{cor}[theorem]{Corollary}
\theoremstyle{plain} 
\theoremstyle{plain} \newtheorem{lsec}[theorem]{Line Segment Extension Conjecture}
\theoremstyle{plain} 
\theoremstyle{definition} 
\theoremstyle{definition} \newtheorem{notation}[theorem]{Notation}
\theoremstyle{remark} 
\theoremstyle{remark} \newtheorem{example}[theorem]{Example}

\DeclareMathOperator{\hdim}{dim_{\mathsf{H}}}
\DeclareMathOperator{\mdim}{dim_{\mathsf{M}}}
\DeclareMathOperator{\pdim}{dim_{\mathsf{P}}}

\DeclareMathOperator{\umdim}{{\overline{dim}_{\mathsf{M}}}}

\newcommand{\R}{\mathbb{R}}
\newcommand{\Rn}{\mathbb{R}^n}

\newcommand{\de}{\delta}
\newcommand{\su}{\subset}
\newcommand{\sm}{\setminus}

\newcommand{\cS}{{\mathcal{S}}}
\newcommand{\cL}{{\mathcal{L}}}
\newcommand{\cC}{{\mathcal{C}}}

\begin{document}

\title{Are lines much bigger than line segments?}

\author{Tam\'as Keleti}

\date{}

\keywords{Hausdorff dimension, lines, union of line segments,
Besicovitch set, Nikodym set, Kakeya Conjecture.}

\subjclass[2010]{28A78}

\address{Institute of Mathematics, E\"otv\"os Lor\'and University, 
P\'az\-m\'any P\'e\-ter s\'et\'any 1/c, H-1117 Budapest, Hungary}

\email{tamas.keleti@gmail.com}

\urladdr{http://www.cs.elte.hu/analysis/keleti}

\thanks{The author was supported by 
Hungarian Scientific Foundation grant no.~104178. 
Part of this research was done while the author was a visiting researcher
at the Alfr\'ed R\'enyi Institute of Mathematics,
whose hospitality is gratefully acknowledged.}

\begin{abstract}
We pose the following conjecture:

\noindent
($\star$) \emph{
If $A$ is the union of line segments in $\Rn$, and $B$ is the  
union of the corresponding full lines then the Hausdorff dimensions 
of $A$ and $B$ agree.}

We prove that this conjecture would imply that every Besicovitch set 
(compact set that contains line segments in every direction) 
in $\Rn$ has Hausdorff dimension at least $n-1$ and (upper) Minkowski dimension $n$.
We also prove that conjecture ($\star$) holds if the Hausdorff dimension of $B$
is at most $2$, so in particular it holds in the plane.
\end{abstract}

\maketitle

\section{Introduction}
Let $A$ be the union of line segments in $\Rn$, and let $B$ be the  
union of the corresponding full lines. Can $B$ be much bigger than $A$?
According to Lebesgue measure \emph{yes}: Nikodym's classical
construction \cite{Ni} 
is a subset $F$ of the unit square $(0,1)^2$ with Lebesgue measure $1$ such 
that 
for each $x\in F$ there is a line $l_x$ through $x$ intersecting $F$ in the 
single point $x$. Thus if we take  inside $(0,1)^2$ an open subsegment 
of each $l_x$ with endpoint $x$
then the union of these open line segments must
have Lebesgue measure zero (since it is contained in $(0,1)^2\sm F$)
but by extending each open line segment only with the endpoints the new union
has Lebesgue measure $1$ (since it contains $F$). In higher dimension the 
line segments can be even pairwise disjoint: Larman \cite{La} constructed a 
compact set $F$ made up of a disjoint union of closed line segments in $\R^3$   
such that the union of the closed segments have positive Lebesgue
measure but the union of the corresponding open segments has Lebesgue measure 
zero.

In this paper we study the question if extending the line segments to full lines
can increase the Hausdorff dimension. We pose the following:

\begin{lsec}
If $A$ is the union of line segments in $\Rn$, and $B$ is the  
union of the corresponding full lines then the Hausdorff dimensions 
of $A$ and $B$ agree.
\end{lsec}

This conjecture turns out to be closely related to the so called Kakeya Conjecture.
Recall that a compact subset of $\Rn$ is called a \emph{Besicovitch set} 
if it contains unit line segment in every direction. 
Besicovitch \cite{Be} constructed in 1919 such sets of Lebesgue measure zero.
The Kakeya Conjecture asserts that every Besicovitch set in $\Rn$ has Hausdorff (or at least upper Minkowski) dimension $n$. Davies \cite{Da} proved this in the plane in 1971. 

Although Kakeya Conjecture is very important in many areas of mathematics, especially in harmonic analysis, because of deep connections with
several important conjectures of different areas 
(see the survey papers \cite{Ta} and \cite{Wo99} about Kakeya Conjecture and these connections), all the known partial results in dimension at least three are much
weaker than the conjecture.

Wolff \cite{Wo} proved in 1995 that the Hausdorff dimension of a Besicovitch set 
in $\Rn$ is at least $\frac{n+2}{2}$, which is still the best estimate for
the Hausdorff dimension for $n=3,4$. 
The current best estimate for $n\ge 5$ is due to Katz and Tao \cite{KaTa} (2002) and is of
the form $An+B$ where $A=2-\sqrt 2 =0.585\ldots$ and $B=4\sqrt 2 - 5 = 0.6\ldots$.
For (upper) Minkowski dimension slightly better estimates exist for $n=3$ (Katz-{\L}aba-Tao \cite{KaLaTa} 2000), for $n=4$ ({\L}aba-Tao \cite{LaTa} 2001) and for $n\ge 24$ 
(Katz and Tao \cite{KaTa} 2002).

Clearly, Line Segment Extension Conjecture would imply that in the Kakeya problem it does
not matter if we consider line segments or full lines. But this would not be very interesting, this is what probably everybody expects, since the constructions for 
Besicovitch sets of zero Lebesgue measure also works for full lines. 
It might be much more surprising that the new conjecture would have not only the above 
consequence but in the following sense it would imply the 
Kakeya Conjecture itself.

\begin{theorem}\label{t:implication}

\begin{itemize}

\item[(i)]
The Line Segment Extension Conjecture for $n$ would imply that every
Besicovitch set in $\Rn$ has Hausdorff dimension at least $n-1$.

\item[(ii)] If the Line Segment Extension Conjecture holds for every $n\ge 2$ then
every Besicovitch set in $\R^n$ has packing and upper Minkowski dimension $n$.

\end{itemize}

\end{theorem}

After seeing these strong consequences we show some evidence in support of 
the new conjecture.

\begin{theorem}\label{t:main}

\begin{itemize}

\item[(i)] The Line Segment Extension Conjecture holds in the plane.

\item[(ii)] More generally, the Line Segment Extension Conjecture holds if the
Hausdorff dimension of the union of the line segments is less than two
or the Hausdorff dimension of the union of the lines is at most two.

\end{itemize}

\end{theorem}

The above theorems are proved in Section~\ref{proofs}. 
In Section~\ref{examples} we present some examples that show that certain
variants of the Line Segment Extension Conjecture are not valid.

\section{Proof of the results}\label{proofs}

\begin{notation}
Let $\cS$ be a collection of line segments. By extending each line segment to a line 
we get a collection of lines that we denote by $\cL(\cS)$.

For any $a,b\in\R^{n-1}$ let $l(a,b)$ denote the line 
$(x_2,\ldots,x_{n-1})=x_1 a+b$. 

Let $\cL$ be a collection of lines in $\Rn$ such that none of them are
orthogonal to the $x_1$ axis. 
Then let
$$
\cC(\cL)=l^{-1}(\cL)=\{(a,b)\in\R^{n-1}\times\R^{n-1}\ : \ l(a,b)\in\cL\}.
$$
Note that $l$ is continuous, so if $\cL$ is Borel then so is $\cC(\cL)=l^{-1}(\cL)$.

Let $\hdim$, $\pdim$, $\mdim$ and $\umdim$ denote the
Hausdorff dimension, the packing dimension, the Minkowski dimension and the 
upper Minkowski dimension, respectively; see the definitions in \cite{Fa90} or \cite{Ma}.
\end{notation}

\begin{proof}[Proof of Theorem \ref{t:implication}]
First we present a well known short argument (see e.g. in \cite{Dv}) that shows that
if for every $n\ge 2$ any Besicovitch set in $\R^n$ has Hausdorff dimension at least
$n-1$ then every Besicovitch set in $\R^n$ has packing and upper Minkowski 
dimension $n$, and so
(i) implies (ii).
We will need the simple observation that the Cartesian product of Besicovitch sets is a 
Besicovitch set and the well known facts (see \cite{Ma}) 
that for any Borel sets $A$ and $B$ we have $\hdim(A)\le \pdim(A) \le \umdim(A)$ and
$\pdim(A\times B)\le\pdim(A)+\pdim(B)$.
Suppose that $B$ is Besicovitch set in $\R^n$ with 
$\pdim(B)<n$. Then for large enough $k$ we have $k\pdim(B)<kn-1$. 
Then $B^k=B\times\ldots\times B$ is a Besicovitch set in $\R^{kn}$ with 
$\hdim(B^k)\le\pdim(B^k)\le k\pdim(B) < kn-1$,
which would contradict our assumption.

Therefore it is enough to prove (i).
The claim of (i) follows from the observation that if we consider $\Rn$ 
inside the $n$-dimensional projective space then lines in a given direction
corresponds to lines through a given point of the "hyperplane at infinity", 
so after applying a projective map that takes the hyperplane at infinity to a
one codimensional plane of 
$\Rn$ then, by extending the line segments of the image of the Besicovitch set
we get a set that contains that hyperplane. 
To make the argument more accessible we present the same argument
more formally, without referring to projective geometry.

For $x\in \R\sm\{0\}, y\in\R^{n-1}$ let $P(x,y)=(\frac1x, \frac yx)$.
This is a locally Lipschitz map, so it preserves Hausdorff dimension.

We claim that $P$ maps the punched line $l(a,b)\sm\{(0,b)\}$ to the punched line
$l(b,a)\sm\{(0,a)\}$. Indeed,
\begin{align*}
P(l(a,b)\sm\{(0,b)\} & 
= P(\{(t,at+b)\ :\ t\in\R\sm\{0\}) \\
&= \left\{\left(\frac1t,\frac{at+b}{t}\right)\ :\ t\in\R\sm\{0\} \right\}\\
& = \{(u,a+ub)\ :\ u\in\R\sm\{0\} \} = l(b,a)\sm\{(0,a)\}.
\end{align*}

If $B$ is a Besicovitch set in $\Rn$ then for every $a\in\R^{n-1}$
 the set $B$ contains a subsegment $s(a)$ of a line $l(a,b)$ for some $b\in\R^{n-1}$
and we can clearly guarantee that $s(a)\su \Rn\sm(\{0\}\times \R^{n-1})$. 
Thus, by the claim of the above paragraph $P(s(a))$ is a line segment and its line
extension contains $(0,a)$. Therefore if $S=\{s(a)\ :\ a\in\R^{n-1}\}$ and 
$S'=\{P(s(a))\ :\ a\in\R^{n-1}\}$ then 
$\cup\cL(\cS')\supset \{0\}\times\R^{n-1}$, so $\hdim\cup\cL(\cS')\ge n-1$.
Since $B\supset\cup\cS$ and $P$ preserves Hausdorff dimension,
we have $\hdim B\ge \hdim(\cup\cS)=\hdim(\cup\cS')$.
Therefore applying the Line Segment Extending Conjecture to $\cS'$
would give $\hdim B\ge n-1$.
\end{proof}

Although in some arguments one might need measurability assumptions,
the following observation shows that we do not need to assume measurability.
The author learned this from M\'arton Elekes \cite{El}, 
and this proof is also essentially due to him.

\begin{lemma}\label{l:Borel}

\begin{itemize}
\item[(i)]
For any collection $\cS$ of closed line segments in $\R^n$ 
there exists a collection $\cS'\supset\cS$ of 
closed line segments
with $\hdim(\cup \cS')=\hdim(\cup \cS)$ such that 
$\cL(\cS')$ is Borel.
\item[(ii)]
If $\cL$ is a Borel collection of lines then
$\cup \cL$ is analytic.
\end{itemize}
\end{lemma}

\begin{proof}
(i) We can suppose that for some fixed $\de>0$ and bounded open set $B\su\R^n$
each $s\in \cS$ is contained in $B$ and has length at least $\de$ since 
we can write $\cS$ as a countable union $\cS=\cup_j \cS^j$ of such subcollections
and if ${\cS^j}'$ is good for $\cS^j$
then $\cup_j{\cS^j}'$ is good for $\cS$.

Let $A=\cup \cS$. Then $A\su B$.
Since for any $s$ any set is contained in a $G_\de$ set of the same  
$s$-dimensional Hausdorff (outer) 
measure (see \cite[471D (b)]{Fr}), we can take a
$G_\de$ set $A'\supset A$ with $A'\su B$ and $\hdim(A')=\hdim(A)$.
Write $A'$ in the form $A'=\cap_{k=1}^\infty G_k$, where each $G_k$ is open
and $B\supset G_1\supset G_2 \supset\ldots$.
Let $\cS_k$ be the collection of those closed line segments inside $G_k$
that has length at least $\de$.
Let $\cS'=\cap_{k=1}^\infty \cS_k$.

Then $\cS'\supset\cS$ since each $\cS_k$ contains $\cS$.
Since for any $k$, $\cup \cS'\su \cup \cS_k\su G_k$ by construction,
we have $\cup\cS'\su\cap_k G_k = A'$, 
so $\hdim \cup\cS \le \hdim \cup\cS' \le \hdim A' = \hdim \cup S$,
hence  $\hdim \cup\cS = \hdim \cup\cS'$.

We claim that $\cL(\cS')=\cap_{k=1}^\infty \cL(\cS_k)$.
Indeed if $l\in\cL(\cS')$ then $l$ is the extension of a line segment 
$s\in\cap_k \cS_k$, so $l\in\cL(\cS_k)$ for each $k$, 
hence $l\in \cap_{k=1}^\infty \cL(\cS_k)$.
Conversely, if $l\in\cap_{k=1}^\infty \cL(\cS_k)$
then for each $k$ the set $l\cap G_k$ contains a closed line segment 
of length at least $\de$.
Since $B$ is bounded, $B\supset G_1\supset G_2\supset\ldots$, 
this implies that there exists a closed line segment $s\su l$
of length at least $\de$ that is contained in every $G_k$, 
so $s\in \cS'$ and $l\in\cL(\cS')$.

Since $G_k$ is open, $\cL(\cS_k)$ is also open, so
$\cL(\cS')=\cap_k \cL(\cS_k)$ is Borel.

(ii)
For $i=1,\ldots,n$ let $\cL_i$ contain those lines of $\cL$ that are not orthogonal
to the $i$-th axis. Then $\cL=\bigcup_{i=1}^n \cL_i$, so it is enough to show
that each $\cup\cL_i$ is analytic. Since each $\cL_i$ is Borel, this means 
that we can assume that $\cL=\cL_i$ for some $i$ and without loss of generality
we can clearly suppose that $i=1$.
Then $\cC(\cL)$ is defined and $\cup\cL=F(\R\times\cC(\cL))$,
where $F(t,(a,b))=(t,ta+b)$, so $\cup\cL$ is the continuous
image of a Borel set, so it is analytic.
\end{proof}

\begin{lemma}\label{l:linesections}
If $\cL$ is a Borel collection of nonvertical lines of the plane then 
the intersection of $\cup\cL$ with almost every vertical line has the same Hausdorff
dimension.
\end{lemma}

\begin{proof}
Let $v_t$ denote the vertical line $x=t$ and let $B=\cC(\cL)$.
One can easily check (or see \cite[proof of 18.11]{Ma})
that $(\cup\cL)\cap v_t=\{t\}\times \pi_t(B)$, where $\pi_t(x,y)=tx+y$, 
so $(\cup\cL)\cap v_t$ is similar to the projection of $B$ to the line $x=ty$.
Then Marstrand's projection theorem \cite{Mar} gives the claim of the Lemma.
\end{proof}

Now we are ready to prove Theorem~\ref{t:main}. First we prove its 
first part and then we show how that implies the more general second part.

\begin{theorem}\label{t:plane}
Let $\cS$ be a collection of line segments in $\R^2$ and 
$\cL(\cS)$ be the collection
of lines we get by extending each line segment of $\cS$. 
Then $\hdim(\cup\cS)=\hdim(\cup\cL(\cS))$.
\end{theorem}

\begin{proof}
We can suppose that each $s\in\cS$ intersects two fixed segments $e$ and $f$ that are opposite sides of a rectangle since we can decompose $\cS$ as a countable union of subcollections with this property and if the result can be applied to each subcollection
then it follows for the union $\cS$.
 By Lemma~\ref{l:Borel} we can also suppose that $\cL(\cS)$ is Borel 
and $\cup \cL(\cS)$ is analytic.

Fix an arbitrary $u$ such that 
$u<\hdim(\cup\cL(\cS))$ and the proof will be completed by showing that
$\hdim(\cup\cS)\ge u$.
If $u=1$ then this is clear, 
so we can suppose that $u>1$. 
Then we can apply  Marstrand's slicing theorem \cite{Mar} (to $\cup\cL(\cS)$),
which states that if $u>1$ and an analytic subset $A$ of the plane has positive 
$u$-dimensional Hausdorff measure
then in almost every direction, positively many lines meet $A$ in a set 
of Hausdorff dimension at least $u-1$.
Therefore we get that 
for almost every unit vector $w$ there exists a set $T\su\R$ of positive
Lebesgue measure such that for any $t\in T$ we have
$\hdim((\cup\cL(\cS))\cap l_{w,t})\ge u-1$, where $l_{w,t}$ is the line 
$\{a\in\R^2:a\cdot w = t\}$.

Choose distinct parallel lines $l_0$ and $l_1$ so that both of them separate 
$e$ and $f$ and they are orthogonal to such a non-exceptional unit vector $w$.
Then every segment of $\cS$ intersects both $l_0$ and $l_1$.
Without loss of generality we can suppose that $w=(1,0)$, $l_0=v_0$ and $l_1=v_1$,
where $v_t$ denotes the vertical line $x=t$. Thus by the choice of $w$ and $T$ we have
$$
\hdim((\cup\cL(\cS))\cap v_t)\ge u-1 \qquad (\forall t\in T).
$$

Since $T$ has positive Lebesgue measure, by Lemma~\ref{l:linesections} we get that
$$
\hdim((\cup\cL(\cS))\cap v_t) \ge u-1 \quad \textrm{for almost every } t\in\R.
$$
Since every segment of $\cS$ intersects both $l_0=v_0$ and $l_1=v_1$
we have
$$
(\cup\cL(\cS))\cap v_t=(\cup\cS)\cap v_t \qquad (\forall t\in[0,1]).
$$
Thus we obtained that 
$$
\hdim((\cup\cS)\cap v_t) \ge u-1 \quad \textrm{for almost every } t\in[0,1],
$$
which implies (see e.g. \cite[Theorem 7.7]{Ma}) 
$\hdim(\cup\cS)\ge u$, which completes the proof.
\end{proof}

\begin{cor}
Let $\cS$ be a collection of line segments in $\R^n$ and 
$\cL(\cS)$ be the collection
of lines we get by extending each line segment of $\cS$. 
Suppose also that
$\hdim(\cup \cS)<2$ or $\hdim(\cup \cL(\cS))\le 2$.
Then $\hdim(\cup\cS)=\hdim(\cup\cL(\cS))$.
\end{cor}

\begin{proof}
Similarly as in the proof of Theorem~\ref{t:plane}, by decomposing $\cS$, we get that
we can suppose that every segment of $\cS$ has roughly the same direction and,
by Lemma~\ref{l:Borel}, we can suppose that $\cup\cL(\cS)$ is analytic.

Let $p_W$ denote the orthogonal projection to the subspace $W$.
By the Marstrand-Mattila projection theorem \cite{Ma75} for almost every $2$-dimensional
subspace $W\su \Rn$ we have
$$
\hdim p_W \left(\cup\cL(\cS)\right) = \min(2,\hdim\cup\cL(\cS)).
$$
Since the segments of $\cS$ have roughly the same direction,
we can fix a plane $W$ with the above property so that the projection of every segment 
of $\cS$ into $W$ is a non-degenerated segment. 
Thus 
$$
p_W (\cup\cS) = \cup \{p_W (s)\ : s\in \cS\} \ \textrm{ and }\ 
p_W (\cup\cL(\cS)) = \cup \cL (\{p_W (s)\ : s\in \cS\}).
$$
Therefore, applying Theorem~\ref{t:plane} in the plane $W$ to the projected
segments we get that
$$
\hdim\left( \cup \{p_W (s)\ : s\in \cS\} \right)=
\hdim\left( \cup \cL(\{p_W (s)\ : s\in \cS\}) \right).
$$

Combining the displayed equalities of the above paragraph with obvious inequalities we get that
$$
\hdim\cup\cL(\cS)\ge \hdim\cup\cS \ge \hdim p_W(\cup\cS)=
\min(2,\hdim\cup\cL(\cS)).
$$
Since $\hdim(\cup \cS)<2$ or $\hdim(\cup \cL(\cS))\le 2$, we get that
$\hdim(\cup\cS)=\hdim(\cup\cL(\cS))$.
\end{proof}

\section{Concluding remarks}\label{examples}

The following example shows that the Minkowski dimension version of the 
Line Segment Extension Conjecture would be false even in the plane.

\begin{example}
Let A be the union of the line segments with endpoints $((2m-1)2^{-n}, 0)$ and 
 $((2m-1)2^{-n}, 2^{-n})$, where $n=1,2,\ldots$ and $m=1,\ldots,2^{m-1}$,
and let $B$ be the set we obtain by extending the line segments of $A$ to unit line
segments in the square $[0,1]\times [0,1]$. 

Then $B$ is dense in $[0,1]\times [0,1]$, so it has Minkowski dimension $2$. 
But the Minkowski dimension of $A$ is $1$, which can be calculated directly or 
by noticing that $A$ is an inhomogeneous self-similar set (see the definition in
\cite{Fra}) and then the formula of \cite[Corollary 2.2]{Fra} can be applied.
\end{example}

We do not know if such an example exists with the extra requirement
that the lengths of the line segments must have a positive lower bound. 
We do not know if for packing dimension the Line Segment Extension Conjecture
would hold at least in the plane.

The next example shows that linearity must be used, 
Line Segment Extension Conjecture would be false for general curves instead of
line segments.

\begin{example}
We construct an embedding of the infinite rooted quadrary tree into $\R^3$. 
Consider the points of the form $((2k-1)2^{-n}, (2l-1)2^{-n}, 2^{-n})$, where
$n=0,1,2,\ldots$ and $k,l=1,\ldots,2^{n-1}$ and join each point
$((2k-1)2^{-n}, (2l-1)2^{-n}, 2^{-n})$ with a line segment to the four
points of the form $((4k-2\pm 1)2^{-n-1}, (4l-2\pm 1)2^{-n-1}, 2^{-n-1})$.

Then the infinite branches of this tree are rectifiable curves. The union of these
curves is a countable union of line segments, so it has Hausdorff dimension $1$.
But if we extend each of these curves with its limit point then we get all points
of the square $[0,1]\times[0,1]\times\{0\}$, so this way we get a set of 
Hausdorff dimension $2$.
%
\end{example}






\end{document}